\documentclass{article}

\usepackage{amsmath,amsthm}
\usepackage{amssymb}
\usepackage{float}
\usepackage{color}
\usepackage{graphicx}

\usepackage{setspace}
\newtheorem{theorem}{Theorem}[section]

\newtheorem{proposition}[theorem]{Proposition}

\newtheorem{remark}[theorem]{Remark}
\newcommand{\ISE}{\operatorname{ISE}}
\newcommand{\id}{\operatorname{I}}

\begin{document}

\title{Averaging of density kernel estimators}
\author{O. Chernova, F. Lavancier and P. Rochet}

\maketitle

\begin{abstract}
We study the theoretical properties of a linear combination of density kernel estimators obtained from different data-driven  bandwidths. The average estimator is proved to be asymptotically as efficient as the oracle, with a control on the error term. The performances are tested numerically, with results that compare favorably to other existing procedures.
\end{abstract}


\section{Introduction}

Kernel estimation is an efficient and commonly used method to estimate a density from a sample of independent identically distributed random variables. It relies on the convolution of the empirical measure with a function $K$ (the kernel), adjusted via a tuning parameter $h$ (the bandwidth). 

Several commonly used data-driven methods for bandwidth selection such as \cite{silverman} or \cite{sheather_jones91} have stood the test of time although none can be recognized as objectively best. In the past decades, aggregation for kernel density estimators has been investigated as an alternative to bandwidth selection. Methods proposed in the literature include stacking (\cite{Smyth1999}), sequential processes (\cite{Catoni97themixture,MR1762904}) and minimization of quadratic (\cite{rigollettsybakov}) or Kullback-Leibler (\cite{butucea2017optimal}) criteria. In these papers, the initial estimators are  assumed non-random, which is generally achieved by dividing the sample to separate training and validation.

The aim of the present  article is to propose a new procedure to combine several competing density kernel estimators obtained from different, possibly data-driven, bandwidths. 
The method,  in the spirit of model averaging,  aims at minimizing the integrated square error of a linear combination of the kernel estimators. In this particular context, the first order asymptotic of the error is known up to a single parameter $\gamma$ equal to the integrated squared second derivative of the density. The easily tractable error is precisely what makes kernel estimation a good candidate for averaging procedures, as we discuss in Section~\ref{sec:somefacts}. Furthermore, the estimation of $\gamma$ can be made from the same data used to estimate the density so that no sample splitting is needed. 
The method is detailed in Section~\ref{sec:averaging}, where it is proved to be asymptotically as efficient as the best possible combination, referred to as the oracle. Our simulation study demonstrates that our method compares favorably to other existing procedures and confirms that sample splitting may lead to poorer results in this setting.

\section{Some facts on kernel estimators}\label{sec:somefacts}

Let $X_1,...,X_n$ be a sample of independent and identically distributed  real random variables with density $f$ with respect to the Lebesgue measure. Given a kernel  $K: \mathbb R \to \mathbb R$ and a bandwidth $h>0$, the kernel  estimator of $f$ is defined as $\hat f_h(x) := (nh)^{-1} \sum_{i=1}^n K \big( h^{-1} (X_i - x) \big), x \in \mathbb R$. Henceforth, we  assume that $K$ is a bounded, symmetric around zero, density function on $\mathbb R$ such that
\begin{equation} \tag{\textbf{H}$_K$}\label{cond_kernel} \textstyle \Vert K \Vert^2 := \int K^2(u) du < \infty \ \text{ and } \ c_K := \int u^2 K(u) du < \infty. \end{equation}
Concerning $f$, we assume  it is twice continuously differentiable on $\mathbb R$ and 
\begin{equation} \tag{\textbf{H}$_f$}\label{cond_f} \textstyle 
\text{$f$, $f'$ and $f''$ are bounded and square integrable.}
\end{equation}

Under \eqref{cond_kernel} and \eqref{cond_f}, \cite{hall1982limit} showed that the Integrated Square Error ($\ISE$) of a kernel estimator $\hat f_h$ satisfies
\begin{equation}\label{eq:ISE} \ISE(\hat f_h) := \Vert \hat f_h - f \Vert^2  = \frac{\Vert K \Vert^2}{nh} + \gamma \frac{h^4 c_K^2}{4} +o_p\Big(\frac{1}{nh}+h^4\Big),\end{equation}
where $\gamma := \int f''(x) ^2dx$. In the typical case where $(nh)^{-1}$ and $h^4$ balance out, 
meaning that $h = O(n^{-1/5})$  and $h \neq o(n^{-1/5})$, or for short $h \asymp n^{-1/5}$,
approximating $\ISE(\hat f_h)$ is achieved from estimating $\gamma$ only. This result is technically only valid for a deterministic bandwidth although in practice, most if not all common methods for bandwidth selection rely on some tuning parameters one has to calibrate from the data. Consequently, the bandwidth $h$ is generally a data-driven approximation of a deterministic one $h^*$, hopefully sharing its asymptotic properties. The effect of this approximation can nevertheless be considered negligible if the integrated square errors are asymptotically equivalent, in the sense that
\begin{equation}
\label{iseeq} \frac{\ISE(\hat f _{h}) - \ISE(\hat f_{h^*}) }{\ISE(\hat f_{h^*})} =o_p(1).\end{equation}

We prove in the next proposition that this equivalence does generally hold true, implying that Hall's result extends to data driven bandwidths. The setting encompasses most common data driven bandwidth selection procedures, as discussed after the proof. We moreover consider not only one but a collection of $k$ data-driven bandwidths $\mathbf h=(h_1,...,h_k)$, in order to study the integrated crossed errors between 
their associated kernel estimator, that are encoded in  the Gram matrix $\Sigma$ with general term $\Sigma_{ij}  = \int \big(\hat f_{h_i}(x) - f(x) \big) \big( \hat f_{h_j}(x) -f(x) \big) dx$.

\begin{proposition}\label{prop:data-driven}
Assume  \eqref{cond_kernel}, \eqref{cond_f}  and further that the kernel $K$ has compact support and is twice continuously differentiable. If there exists a deterministic bandwidth $h_i^* \asymp n ^{-1/5}$ for all data-driven bandwidth $h_i$ satisfying $h_i - h^*_i = o_p(n^{-2/5})$, then
\begin{equation} 
\label{eq:Sigma_inf} \Sigma = A + \gamma B +o_p (n^{-4/5} ),\end{equation}
where $A_{ij} = \frac 1 n \int K( u/h_i ) K (u/h_j) du $ and $B_{ij}:= \frac 1 4 h_i^2h_j^2 c_K^2$, $i,j=1,...,k$.
\end{proposition}
 
\begin{proof}
 Following \cite{hall1987extent}, let $\Delta(h)=\ISE(\hat f _{h})$ and consider the first-order expansion $\Delta(h_i)-\Delta(h^*_i) = (h_i-h^*_i)\Delta'(\tilde h_i)$, for some $\tilde h_i$ between $h_i$ and $h^*_i$. From Section~2 and Lemma~3.2 in  \cite{hall1987extent}, we know that  $\Delta'(\tilde h_i)=O_p(n^{-3/5})$. Combined with the fact that $\Delta(h^*_i)  \asymp O_p(n^{-4/5})$, the condition $h_i - h^*_i = o_p(n^{-2/5})$ implies $\big( \ISE(\hat f _{h_i}) - \ISE(\hat f_{h_i^*}) \big) /\ISE(\hat f_{h_i^*}) =o_p(1)$ for all $i$. Using  the arguments of Theorem 2 in \cite{hall1982limit}, we get
$ \Sigma = A^* + \gamma B^* +o_p (n^{-4/5} )$, where $A^*$ and $B^*$ are defined similarly as $A$ and $B$ with $h_i^*$ in place of  $h_i$. The map $(x,y) \mapsto \int K(u/x) K(u/y) du$ is continuous  for $x,y > 0$, which implies its uniform continuity on every compact set in $(0,+ \infty)^2$. Applying this function to  the sequences $(x^*_n,y^*_n) = (n^{1/5} h^*_i, n^{1/5} h^*_j), \ i \neq j$, which are bounded away from zero, and $(x_n,y_n) = (n^{1/5} h_i, n^{1/5} h_j)$, we deduce that $ n^{4/5} (A - A^*) = o_p(1)$. Similarly, $n^{4/5} (B - B^*) = o_p(1)$ yielding the result.
\end{proof}

The most common bandwidth selection procedures do verify the condition $h_i - h^*_i = o_p(n^{-2/5})$ for some deterministic $h^*$ (see \cite{jones1996brief}), making the approximation \eqref{eq:Sigma_inf} available. For instance, Silverman's rule of thumb approximates the
	deterministic bandwidth $h^*= c \min\{\sigma, \textsc{iqr}/1.34\}\, n^{-1/5}$ where $\sigma$ is the standard deviation, {\sc iqr} the inter-quartile range and $c$ is either equal to $0.9$ or $1.06$, based on empirical considerations, see \cite{silverman}. The biased and unbiased least-square cross-validation bandwidths discussed in \cite{scott1987biased,hall1987extent} and the plug-in approach of \cite{sheather_jones91} approximate the deterministic bandwidth $h^*= \Vert K \Vert^{2/5} (n c_K \gamma )^{-1/5}$. The latter achieves a rate $h/h^* -1= O_p(n^{-5/14})$ that can be improved up to $O_p(n^{-1/2})$ if $\gamma$ is estimated following \cite{hall1991optimal}. Note finally that, as argued by several authors, a truncation argument allows to extend Proposition~\ref{prop:data-driven} to non compactly supported kernels $K$, see e.g. \cite{hall1987amount} or Remark 3.9 in \cite{park90}.

\section{The average estimator}\label{sec:averaging}

Let $\mathbf h = (h_1,...,h_k)^\top \in \mathbb R_+^k$ be a collection of (possibly data-driven) bandwidths and set $\mathbf {\hat f} = (\hat f_{h_1},...,\hat f_{h_k})^\top$. Following \cite{lavancier2016general}, we consider an estimator of $f$ expressed as a linear combination of the $\hat f_{h_i}$'s,
\begin{equation}\label{eq:averaging}  {\hat f}_\lambda = \lambda^\top \mathbf{\hat f} =  \sum_{i=1}^k \lambda_j \hat f_{h_i}, \end{equation}
where the weight vector $\lambda=(\lambda_1,...,\lambda_k)^\top$ is constrained to sum up to one, i.e.~$\lambda^\top \mathbf 1 = 1$ for $\mathbf 1 = (1,...,1)^\top$. Under this normalizing constraint, the integrated square error of $\hat{ f}_{\lambda}$ has the simple expression $ \ISE \big({\hat f}_\lambda \big) = \lambda^\top \Sigma \lambda$. If $\Sigma$ is invertible (which we shall assume throughout), the optimal weight vector $\lambda^*$ minimizing the $\ISE$ under the constraint $\lambda^\top \mathbf 1=1$, is given by $\lambda^* = \big( \mathbf 1^\top \Sigma^{-1} \mathbf 1 \big)^{-1} \Sigma^{-1} \mathbf 1 $. The resulting average estimator $\hat f^* = \lambda^{*\top} \hat{\mathbf f}$ is called the oracle.

With all bandwidths $h_i$ of order $n^{-1/5}$, we know from Proposition~\ref{prop:data-driven} that $\Sigma = A + \gamma B + o_p (n^{-4/5})$. Because both $A$ and $B$ are known, approximating $\Sigma$ is reduced to estimating $\gamma = \int f''(x)^2 dx$. This problem has been tackled in the literature, see for instance \cite{hall1987estimation,hall1991optimal,sheather_jones91}. Hence, given an estimator $\hat \gamma$ of $\gamma$, one obtains an approximation of $\Sigma$ by $ \widehat \Sigma= A + \hat \gamma B$. Replacing $\Sigma$ by its approximation $\widehat \Sigma$ yields the average density estimator 
$ \hat f_{AV} = \hat f_{\hat \lambda} = \big( \mathbf 1^\top \widehat \Sigma^{-1} \mathbf 1 \big)^{-1} \mathbf 1^\top\widehat \Sigma^{-1}\mathbf{\hat f}$.

\begin{theorem}\label{thise} Under the assumptions of Proposition~\ref{prop:data-driven}, if $\Sigma$ and $\widehat \Sigma$ are invertible and $\hat \gamma- \gamma = o_p(1)$, then 
$$  \ISE \big(\hat f_{AV} \big) =  \ISE (\hat f^*)  \big(1+ o_p(1) \big).  $$
\end{theorem}

\begin{proof}
Write
$$  \ISE(\hat f_{AV}) = \hat \lambda^\top \Sigma \hat \lambda = \hat \lambda^\top \widehat \Sigma \hat \lambda + \hat \lambda^\top \big(\Sigma - \widehat \Sigma \big) \hat \lambda.$$
By construction, $\hat \lambda^\top \widehat \Sigma \hat \lambda \leq \lambda^{*\top} \widehat \Sigma \lambda^* =  \lambda^{*\top} \Sigma \lambda^* + \lambda^{*\top} \big( \widehat \Sigma - \Sigma \big)\lambda^*$. Moreover, denoting by  $||| . |||$ the operator norm, $|||A||| = \sup_{|| x || =1} || A x ||$, we have for all $\lambda \in \mathbb R^k$, $ |\lambda^\top \big(\Sigma - \widehat \Sigma \big) \lambda | \leq ||| \id -\widehat \Sigma \Sigma^{-1} ||| \  \lambda^\top \Sigma \lambda,$ see the proof of Lemma A.1 in \cite{lavancier2016general}. Applying the above inequality to $\hat \lambda$ and $\lambda^*$, we get
\begin{equation}\label{ineq_oracle} \big( 1 -  ||| \id - \widehat \Sigma \Sigma^{-1} ||| \big) \ \ISE(\hat f_{AV})  \leq \big( 1 +  ||| \id - \widehat \Sigma \Sigma^{-1} ||| \big)  \   \ISE (\hat f^*)  \end{equation}
where we recall $\ISE (\hat f^*)=\lambda^{*\top} \Sigma \lambda^*$. It remains to show $||| \id - \widehat \Sigma \Sigma^{-1} ||| = o_p(1)$. By Proposition~2.1, 
$\Sigma  =A + \gamma B + C$ with $A=O_p(n^{-4/5})$, $B=O_p(n^{-4/5})$ and $C=o_p(n^{-4/5})$. Therefore  
$$ \widehat \Sigma \Sigma^{-1} =  (A + \hat \gamma B ) \Sigma^{-1} = I - C\Sigma^{-1} + (\hat\gamma -\gamma) B \Sigma^{-1}$$
and since $B \Sigma^{-1}=O_p(1)$,
\begin{equation}\label{ineq_key} ||| \id - \widehat \Sigma \Sigma^{-1} ||| \leq  |||C\Sigma^{-1}||| +  |\hat \gamma - \gamma |O_p(1). \end{equation}
The result follows from the fact that $C\Sigma^{-1}=o_p(1)$ and $\hat \gamma - \gamma = o_p(1)$.
\end{proof}

\begin{remark} In our setting, the number $k$ of initial estimators is assumed fixed  although the result remains valid if $k=k_n$ increases slowly with $n$. As seen in the proof, the ISE of $\hat f_{AV}$ approaches that of the oracle $\hat f^*$ provided that $||| \id - \widehat \Sigma \Sigma^{-1} ||| = o_p(1)$. This can still be achieved if $k_n$ increases sufficiently slowly with $n$, e.g.~logarithmically. In practice however, the numerical study shows that the results are less satisfactory with a too large number of initial estimators, due to $\Sigma$ being close to singular. For better performances, we suggest to use no more than four initial estimators, obtained from different methods, in order to reduce linear dependencies (see the discussion in Section \ref{sec:simus}).
\end{remark}

One may be interested in setting additional constraints on the weights $\lambda_i$, restricting $\lambda$ to a proper subset $\Lambda \subset \{\lambda: \lambda^\top \mathbf 1=1\}$. A typical example is to impose the $\lambda_i$'s to be non-negative, a framework usually referred to as convex averaging. In fact, the same result as in Theorem~\ref{thise} holds for any such subset $\Lambda$, using the corresponding oracle and average estimator,  the proof being identical. A reason for considering additional constraints on $\lambda$ is to aim for a more stable solution, which may be desirable in practice especially when working with small samples (see e.g. Table~\ref{table} in Section~\ref{sec:simus}). However, since the oracle is necessarily worse (in term of integrated square error) for a proper subset $\Lambda$, the result lacks a theoretical justification for using a smaller set. Note that, on the contrary, the constraint $\lambda^\top \mathbf 1=1$ is necessary for the equality ISE$(\hat f_\lambda) = \lambda^\top \Sigma \lambda$ to hold true. \\

The next proposition establishes a rate of convergence in the case where the bandwidths $h_i$ used to build the experts $\hat f_{h_i}$ are deterministic and of the order $h_i \asymp n ^{-1/5}$. The additional assumption $\hat \gamma - \gamma =o_p(n^{-2/5})$ is mild as the best known convergence for an estimator $\hat\gamma$ is $\hat \gamma - \gamma =O_p(n^{-1/2})$, see for instance \cite{hall1991optimal}.

\begin{proposition}\label{prop:rate}
Assume \eqref{cond_kernel} and \eqref{cond_f}. If the bandwidths $h_i$ are deterministic with $h_i \asymp n ^{-1/5}$, $\Sigma$ and $\widehat \Sigma$ are invertible and $\hat \gamma - \gamma =o_p(n^{-2/5})$, 
$$ \ISE(\hat f_{AV}) = \ISE (\hat f^* ) + O_p(n^{-6/5}).$$
\end{proposition}

\begin{proof}
Under the assumptions, Theorem~2.1 applies with second order asymptotic expansion $ C = \Sigma - A - \gamma B = O_p(n^{-6/5})$. In view of  \eqref{ineq_oracle} and \eqref{ineq_key}, the rate of convergence for $\ISE(\hat f_{AV}) - \ISE (\hat f^* )$ follows from investigating $|||C\Sigma^{-1}|||$ and $|\hat \gamma - \gamma |$.
Here, $|||C\Sigma^{-1}|||=O_p(n^{-2/5})$ while $|\hat \gamma - \gamma |$ is negligible in comparison by assumption.
\end{proof}

The result of Proposition~\ref{prop:rate}  improves on the residual term $O(n^{-1})$ obtained in \cite{rigollettsybakov} where the initial estimators, or experts, are built from a training sample of size $n_{tr}$, while the aggregation is performed on an independent validation sample of size $n_{va}$ with $n=n_{tr}+ n_{va}$. In fact, \cite{rigollettsybakov} show that conditionally to the training sample (making the experts built once and for all), their aggregation procedure reaches the minimax rate $O(n_{va}^{-1})$, which is at best of the order $O(n^{-1})$.  In our setting, the rate of the residual term is improved due to the initial kernel estimators contributing a factor $O_p(n^{-4/5})$.

\section{Simulations}\label{sec:simus}

Based on a sample of $n$ independent and identically distributed observations, we consider the estimation of the following density functions, depicted in Figure~\ref{fig:dens}:
 the standard normal distribution $\mathcal N(0,1)$;
the Gamma distribution with shape parameter 2 and scale parameter 1;
the Cauchy distribution; the equiprobable mixture of   $\mathcal N(-1.5,1)$ and  $\mathcal N(1.5,1)$; and 
 the mixture of  $\mathcal N(-1.5,1)$ with probability $0.7$ and  $\mathcal N(1.5,1)$ with probability $0.3$. 

\begin{figure}

\label{fig:dens}
	\centering
		\includegraphics[width=\textwidth]{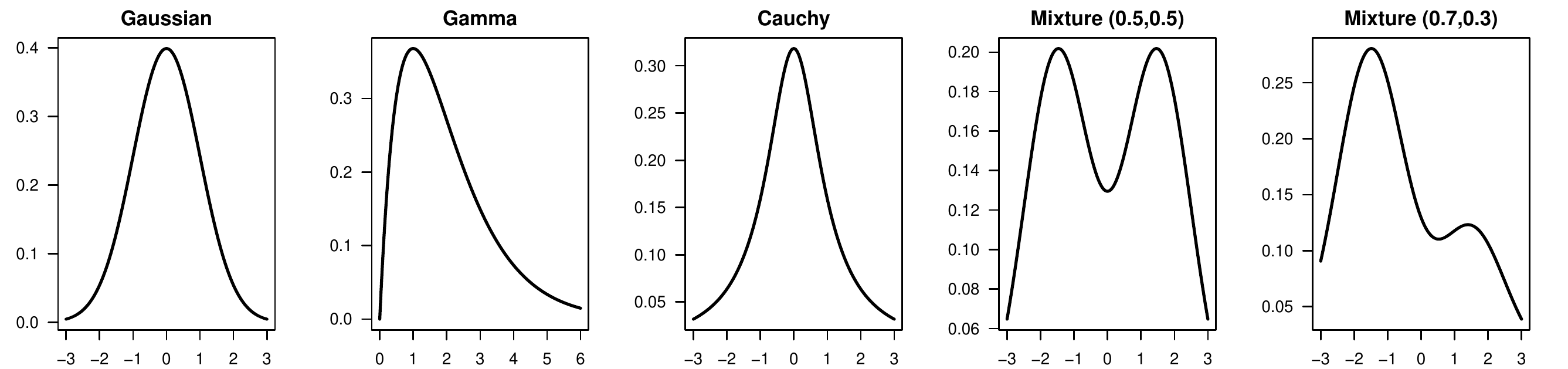}
	\caption{Densities functions considered in the numerical examples}
	
\end{figure}

The initial kernel estimators are built with Gaussian kernel and data-driven  bandwidths \texttt{nrd0} (Silverman's rule of thumb), \texttt{nrd} (its variation with normalizing constant 1.06), and \texttt{SJ} (the plug-in approach of Sheater and Jones), following the default choices in the \texttt{R} software \cite{R}.
The least-square cross-validation bandwidths (\texttt{ucv} and \texttt{bcv} in \texttt{R}) are deliberately not included because they approximate the same deterministic bandwidth $h^*$ as Sheater and Jones' method, which would result in an asymptotically degenerated (non-invertible) matrix $\Sigma$. This was confirmed by simulations (not displayed here), where the inclusion of these estimators did not improve the performances described below. The three kernel estimators are then combined by our method where $\hat \gamma$ is estimated as in \cite{hall1991optimal}. We also assess convex averaging where, in addition, the weights $\lambda_i$ are restricted to non-negative values. For the sake of comparison with existing techniques, we implement the linear and convex aggregation methods considered in \cite{rigollettsybakov} who also use a quadratic loss function. 
In their setting, the experts $\hat f_{h_i}$ are computed from a training sample of half size, independent from the remaining validation sample on which the weights $\lambda_i$ are estimated. In the same spirit, we have also tested this splitting scheme for our average estimator, where $\hat\gamma$ is estimated from the validation sample. For robustness, it is advised in  \cite{rigollettsybakov} to average different aggregation estimators obtained over multiple sample splittings. We followed this recommendation and we considered 10 independent splittings into two samples of equal size.

\begin{table}[H]
\centering
\resizebox{0.9\textwidth}{!} {
\begin{tabular}{|lllllllllll|}
  \hline 
$n$ & Law & nrd & nrd0 &  SJ & . & AV & AVsplit & RT & AVconv & RTconv \\ 
  \hline
50 & Norm & 1936 & 1737 & 1902 & . & 1788 & 1698 & 2480 & 1844 & 2030 \\ 
   & Gamma & 1822 & 1903 & 1864 & . & 1897 & 2088 & 2685 & 1841 & 2173 \\ 
   & Cauchy & 1214 & 1289 & 1233 & . & 1292 & 1493 & 1778 & 1218 & 1452 \\ 
   & Mix05 & 999 & 1043 & 1056 & . & 1239 & 1397 & 1393 & 1063 & 1164 \\ 
   & Mix03 & 1086 & 1145 & 1155 & . & 1238 & 1401 & 1582 & 1159 & 1267 \\ \hline
  100 & Norm & 1057 & 957 & 1026 & . & 962 & 947 & 1325 & 998 & 1113 \\ 
   & Gamma & 1129 & 1239 & 1146 & . & 1153 & 1305 & 1615 & 1135 & 1376 \\ 
   & Cauchy & 748 & 848 & 756 & . & 775 & 906 & 1059 & 750 & 938 \\ 
   & Mix05 & 634 & 699 & 677 & . & 796 & 922 & 888 & 705 & 772 \\ 
   & Mix03 & 669 & 751 & 703 & . & 728 & 852 & 898 & 717 & 823 \\ \hline
  200 & Norm & 628 & 578 & 616 & . & 568 & 556 & 767 & 597 & 660 \\ 
   & Gamma & 701 & 795 & 705 & . & 701 & 772 & 944 & 696 & 839 \\ 
   & Cauchy & 462 & 540 & 454 & . & 452 & 531 & 613 & 454 & 585 \\ 
   & Mix05 & 391 & 453 & 409 & . & 474 & 559 & 501 & 452 & 492 \\ 
   & Mix03 & 398 & 470 & 406 & . & 395 & 467 & 486 & 412 & 510 \\ \hline
  500 & Norm & 310 & 286 & 299 & . & 276 & 274 & 346 & 292 & 321 \\ 
   & Gamma & 388 & 462 & 370 & . & 367 & 411 & 459 & 369 & 459 \\ 
   & Cauchy & 232 & 283 & 220 & . & 208 & 240 & 294 & 221 & 293 \\ 
   & Mix05 & 210 & 253 & 209 & . & 223 & 258 & 227 & 240 & 264 \\ 
   & Mix03 & 223 & 271 & 218 & . & 199 & 222 & 234 & 222 & 282 \\ \hline
  1000 & Norm & 183 & 171 & 177 & . & 163 & 163 & 196 & 174 & 191 \\ 
   & Gamma & 231 & 285 & 216 & . & 211 & 240 & 262 & 215 & 272 \\ 
   & Cauchy & 145 & 182 & 133 & . & 121 & 138 & 178 & 134 & 181 \\ 
   & Mix05 & 126 & 158 & 120 & . & 120 & 138 & 120 & 134 & 159 \\ 
   & Mix03 & 132 & 165 & 125 & . & 108 & 117 & 124 & 127 & 164 \\ \hline
  2000 & Norm & 111 & 104 & 106 & . & 99 & 98 & 113 & 105 & 114 \\ 
   & Gamma & 146 & 183 & 132 & . & 130 & 147 & 160 & 132 & 167 \\ 
   & Cauchy & 84 & 108 & 76 & . & 66 & 73 & 110 & 76 & 101 \\ 
   & Mix05 & 79 & 100 & 73 & . & 68 & 74 & 68 & 79 & 95 \\ 
   & Mix03 & 77 & 98 & 72 & . & 59 & 61 & 66 & 72 & 92 \\ 
   \hline
\end{tabular}}
\caption{{\scriptsize Estimated MISE (based on $10^3$ replications) of the kernel estimators with bandwidths \texttt{nrd}, \texttt{nrd0} or \texttt{SJ} (by default in \texttt{R}) and the combinations of these estimators by our method (AV), our method with sample splitting (AVsplit), the linear method in \cite{rigollettsybakov} (RT), our convex method (AVconv) and the convex method in  \cite{rigollettsybakov} (RTconv).}}
\label{table}
\end{table}

The mean integrated square errors of the aforementioned  estimators are summarized in Table~\ref{table}, depending on the sample size $n$. These errors are approximated by the average over $10^3$ replications of the integrated square errors. It shows that our averaging procedure (AV in the table) outperforms every single initial kernel estimators when the sample size is large ($n\geq 500$) and the gain becomes significant when $n\geq 1000$. On the contrary, our averaging procedure is inefficient for small sample sizes ($n=50$), which is probably explained by a poor use of the asymptotic expansion of $\Sigma$ in this case. In fact, the convex averaging procedure (AVconv in the table) seems preferable for small $n$ although it also fails to achieve the same efficiency as the best estimator in the initial collection.  
A transition seems to occur for moderate sample sizes around $n=100$, where the results of the average estimator are comparable to the best kernel estimator. In all cases, our averaging procedure outperforms the alternative aggregation method of \cite{rigollettsybakov}. Finally, according to the numerical results, a splitting scheme for our method  (AVsplit in the table) is not to be recommended, suggesting that all the available data should be used both for the initial estimators and for $\hat \gamma$, 
which is in line with our theoretical findings.

\bibliographystyle{alpha}      
\bibliography{aggreg}

\end{document}